\documentclass[12pt]{article}
\usepackage{setspace} 
\usepackage[titletoc,title]{appendix}
\usepackage[dvips]{graphicx} 
\usepackage{amsmath,amsthm,amsfonts}
\usepackage{epsfig}
\usepackage{subfig}
\RequirePackage[colorlinks,citecolor=blue,urlcolor=blue,linkcolor=blue]{hyperref}
\hypersetup{
colorlinks = true,
citecolor=blue,
urlcolor=blue,
linkcolor=blue,
pdfauthor = {Daniel Hackmann},
pdfkeywords = { },
pdftitle = {Simulation of paths},
pdfpagemode = UseNone
}
\usepackage{setspace}
\usepackage{amsthm}
\usepackage{subfig}
\usepackage{appendix}
\usepackage{enumerate}
\usepackage{changepage}
\graphicspath{{./}}
\numberwithin{equation}{section} 
    \def\qed{\hfill$\sqcap\kern-8.0pt\hbox{$\sqcup$}$\\}
    \def\beq{\begin{eqnarray}}
    \def\eeq{\end{eqnarray}}
    \def\beqq{\begin{eqnarray*}}
    \def\eeqq{\end{eqnarray*}}

    \def\p{{\mathbb P}}
    \def\e{{\mathbb E}}
    
    \def\r{{\mathbb R}}

    \def\d{{\textnormal d}}

    \def\ind{{\mathbb I}}

    	\def\edist{\,{\buildrel d \over =}\ }
    	\def\var{\textnormal{Var}}
    	\def\cov{\textnormal{Cov}}
    	\def\rz{\mathbb{R}\backslash\{0\}}
    	\def\rdz{\mathbb{R}^d\backslash\{\mathbf{0}\}}
    	\def\bfx{\mathbf{x}}
	\newtheorem{theorem}{Theorem}
	\newtheorem{lemma}{Lemma}
	\newtheorem{proposition}{Proposition}
	\newtheorem{corollary}{Corollary}
	\theoremstyle{definition}

	\newtheorem{remark}{Remark}

\title{Karhunen--Lo\`{e}ve expansions of L\'{e}vy processes}
\author{
{Daniel Hackmann
\footnote{Institute of Financial Mathematics and Applied Number Theory, Johannes Kepler University,
Linz, Austria.  
E-mail: daniel.hackmann@jku.at}}
 }
 
 \date{\today}

\begin{document}

\maketitle

\begin{abstract}
\noindent Karhunen--Lo{\`e}ve expansions (KLE) of stochastic processes are important tools in mathematics, the sciences, economics, and engineering. However, the KLE is primarily useful for those processes for which we can identify the necessary components, i.e., a set of basis functions, and the distribution of an associated set of stochastic coefficients. Our ability to derive these components explicitly is limited to a handful processes. In this paper we derive all the necessary elements to implement the KLE for a square-integrable L\'evy process. We show that the eigenfunctions are sine functions, identical to those found in the expansion of a Wiener process. Further, we show that stochastic coefficients have a jointly infinitely divisible distribution, and we derive the generating triple of the first $d$ coefficients. We also show, that, in contrast to the case of the Wiener process, the coefficients are not independent unless the process has no jumps. Despite this, we develop a series representation of the coefficients which allows for simulation of any process with a strictly positive L\'evy density. We implement our theoretical results by simulating the KLE of a variance gamma process.  
\end{abstract}

\section{Introduction}
Fourier series are powerful tools in mathematics and many other fields. The Karhunen-Lo{\`e}ve theorem (KLT) allows us to create generalized Fourier series from stochastic processes in an, in some sense, optimal way. Arguably the most famous application of the KLT is to derive the classic sine series expansion of a Wiener process $W$ on $[0,1]$. Specifically, 
\begin{align}\label{eq:wklt}
W_t = \sqrt{2}\sum_{k \geq 1}Z_k \frac{\sin\left(\pi(k-\frac{1}{2}t)\right)}{\pi\left(k - \frac{1}{2}\right)}
\end{align}
where convergence of the series is in $L^2(\Omega,\p)$ and uniform in $t \in [0,1]$, and the $\{Z_k\}_{k\geq 1}$ are i.i.d. standard normal random variables. The main result of this paper is to show that a square integrable L\'evy process admits a similar representation as a series of sine functions; the key difference is that the stochastic coefficients  are no longer normal nor independent.\\ \\
The KLT applies much more generally and is thus an important tool in many fields. For example, we see applications of the KLT and Principal Component Analysis, its discrete time counterpart, in physics and engineering \cite{Ghanesh,Phoon}, \cite[Chapter 10]{deep}, in signal and image processing \cite{unser1998wavelets}, \cite[Chapter 1]{sigbook}, in finance and economics \cite{Benko,cont2002dynamics,gun1} and other areas. For interesting recent theses on the KLT from three different points of view see also \cite{shijin} (probability and time series), \cite[Chapter 7]{Luo} (stochastic partial differential equations), and \cite{wang} (statistics).\\ \\
Deriving the Karhunen--L{\`o}eve expansion (KLE) of the type \eqref{eq:wklt} for a square integrable stochastic process $X$ on $[a,b]$ requires two steps: first, one must solve a Fredholm integral equation to obtain the basis functions $\{e_k\}_{k\geq1}$ (c.f. the sine functions in Equation \ref{eq:wklt}). Second, one must identify the distribution of the stochastic coefficients 
\begin{align}\label{eq:stocoef}
Z_k := \int_a^bX_te_k(t)\d t,\quad k\in \mathbb{N}.
\end{align}
In general, obtaining both the basis functions and the distribution of the stochastic coefficients is not an easy task, and we have full knowledge in only a few specific cases. Besides the Wiener process, the Brownian Bridge process, the Anderson--Darling process, and spherical fractional Brownian Motion (see \cite{Istas20061578} for the latter) are some examples. For further examples with derivation see \cite[Chapter 1]{shijin}. Non-Gaussian processes pose an additional challenge and the problem of deriving the KLE is usually left to numerical means (see e.g., \cite{Phoon}). \\ \\
In this paper we derive all the elements of the KLE for a square integrable L\'{e}vy process on the interval $[0,T]$. The result is timely since in many of the fields mentioned above, especially in finance, but recently also in the area of image/signal processing (see e.g., \cite{bouya}), L\'{e}vy models are becoming increasingly popular. In Section \ref{sec:klt} we show that the basis functions are sine functions, identical to those in \eqref{eq:wklt}, and that the first $d$ stochastic coefficients are jointly distributed like an infinitely divisible (ID) random vector. We identify the generating triple of this vector from which it follows that the coefficients are independent only when the process has no jumps, i.e., when the process is a scaled Wiener process with drift. Although simulating dependent multivariate random variables from a characteristic function is generally difficult, in Section \ref{sec:shotnoise} we derive a shot-noise (series) representation for
\begin{align}\label{eq:vecstoco}
Z^{(d)} := (Z_1,Z_2,\ldots,Z_d)^{\text{\textbf{T}}},\quad d \in \mathbb{N},
\end{align}
for those processes which admit a strictly positive L\'{e}vy density. This result, in theory, allows us to simulate the truncated KLE for a large class of L\'{e}vy models. We conclude by generating some paths of a $d$-term KLE approximation of a variance gamma process.\\ \\
To begin, we recall the necessary facts from the theory of L\'{e}vy processes and ID random vectors.

\section{Facts from the theory of L\'{e}vy processes}\label{sec:introlev}
The L\'{e}vy-Khintchine theorem states that every $d$-dimensional ID random vector $\xi$ has a Fourier transform of the form
\begin{align*}
\e[e^{\i\langle \mathbf{z},\xi\rangle}] = e^{-\Psi(\mathbf{z})},\quad \mathbf{z} \in \r^d,
\end{align*}
where
\begin{align}\label{eq:char}
\Psi(\mathbf{z}) = \frac{1}{2}\mathbf{z}^{\textbf{T}}Q\mathbf{z}  - \i\langle \mathbf{a}, \mathbf{z} \rangle - \int_{\r^d\backslash\{\mathbf{0}\}}e^{\i\langle \mathbf{z}, \mathbf{x} \rangle} - 1 - \i\langle \mathbf{z},\mathbf{x} \rangle h(\mathbf{x})\nu(\d \mathbf{x}),
\end{align}
and where $\mathbf{a} \in \r^d$,  $Q$ is a positive semi-definite matrix, and $\nu(\d \mathbf{x})$ is a measure on $\r^d\backslash\{\mathbf{0}\}$ satisfying
\begin{align}\label{eq:int_cond}
\int_{\r^d\backslash\{\mathbf{0}\}}\min(1,\vert \mathbf{x} \vert^2)\nu(\d  \mathbf{x} ) < \infty.
\end{align}
The function $h$ is known as the cut-off function; in general, we need such a function to ensure convergence of the integral. An important fact is that up to a choice of $h$, the generating triple $(\mathbf{a},Q,\nu)$ uniquely identifies the distribution of $\xi$. The L\'{e}vy-Khintchine theorem for L\'{e}vy processes gives us an analogously powerful result, specifically, for any $d$-dimensional L\'{e}vy process $X$ we have
 \begin{align*}
\e[e^{i\langle \mathbf{z},X_t\rangle}] = e^{-t\Psi(\mathbf{z})},\quad \mathbf{z} \in \r^d,\,t \geq 0,
\end{align*}
where $\Psi$ is as in \eqref{eq:char} and $X$ is uniquely determined, up to identity in distribution, by the triple $(\mathbf{a},Q,\nu)$. Following convention, we will refer to the function $\Psi$ as the characteristic exponent of $\xi$ (resp. $X$) and will write $\Psi_{\xi}$ (resp. $\Psi_{X}$) if there is the potential for ambiguity. In one dimension we will write $(a,\sigma^2,\nu)$ for the generating triple; the measure $\nu$ will always be referred to as the L\'{e}vy measure. When $\nu(\d x) = \pi(x)\d x$ for some density function $\pi$, we will write $(a,\sigma^2,\pi)$ and refer to $\pi$ as the L\'{e}vy density. If we wish to be specific regarding the cut-off function we will write $(\mathbf{a},Q,\nu)_{h\equiv \cdot}$ or $(a,\sigma^2,\nu)_{h\equiv\cdot}$ for the generating triples. \\ \\
\noindent In this article we will work primarily with one dimensional L\'{e}vy processes having zero mean and finite second moment; by this we mean that $\e[X_t] = 0$ and $\e[X_t^2] < \infty$ for every $t \geq 0$. We will denote the set of all such L\'{e}vy processes by $\mathcal{K}$. One may show that the later condition implies that $\Psi$ is twice differentiable. Thus, when we work with a process $X \in \mathcal{K}$, we can express the variance of $X_t$ as
\begin{align*}
\var(X_t) = \e[X_t^{2}] = \Psi''(0)t,
\end{align*}
and the covariance of $X_t$ and $X_s$ as
\begin{align*}
\cov(X_s,X_t) = \e[X_sX_t] = \Psi''(0)\min(s,t).
\end{align*}
For notational convenience we will set $\alpha := \Psi''(0)$. \\ \\
\noindent The existence of moments for both L\'{e}vy processes and ID random vectors can be equivalently expressed in terms of the L\'{e}vy measure. An ID random vector $\xi$ or L\'{e}vy process $X$ with associated L\'{e}vy measure $\nu$ has a finite second moment (meaning the component-wise moments) if, and only if,
\begin{align}\label{eq:conditionA}
\int_{\vert \mathbf{x}\vert > 1} \vert \mathbf{x}\vert^2\nu(\d \mathbf{x}) < \infty \tag{Condition A}.
\end{align}
We will denote the class of ID random vectors with zero first moment and finite second moment by $\mathcal{C}$. The subset of $\mathcal{C}$ which also satisfies
\begin{align}\label{eq:conditionB}
\int_{\vert \mathbf{x}\vert \leq 1} \vert \mathbf{x}\vert\nu(\d \mathbf{x}) < \infty. \tag{Condition B}
\end{align}
will be denoted $\mathcal{CB}$ and $\mathcal{KB}$ will denote the analogous subset of $\mathcal{K}$. We remark that any $\xi \in \mathcal{C}$ (resp. $X \in \mathcal{K}$) necessarily has a representation of the form $(\mathbf{0},Q,\nu)_{h\equiv 1}$ (resp. $(0,\sigma^2,\nu)_{h\equiv 1}$). Additionally, any $d$-dimensional $\xi \in \mathcal{CB}$ necessarily has representation $(\mathbf{a},Q,\nu)_{h\equiv 0}$ where $\mathbf{a}$ has entries
\begin{align*}
-\int_{\rdz}P_k(\mathbf{x})\nu(\d \bfx),\quad k \in \{1,2,\ldots d\}
\end{align*}
and $P_k$ is the projection onto the $k$-th component. Analogously, if $X \in \mathcal{KB}$ then we have representation $(a,\sigma^2,\nu)_{h\equiv 0}$ where $a = -\int_{\rz}x\nu(\d x)$.
\section{The Karhunen--Lo\`{e}ve theorem}\label{sec:klt}
Given a real valued continuous time stochastic process $X$ defined on an interval $[a,b]$ and an orthonormal basis $\{\phi_k\}_{k \geq 1}$ for $L^2([a,b])$ we might try to express $X$ as a generalized Fourier series
\begin{align}\label{eq:eigen}
X_t = \sum_{k=1}^\infty Y_k\phi_k(t),\quad\text{ where }\quad Y_k := \int_a^bX_t\phi_k(t)\d t.
\end{align}
In this section, our chosen basis will be derived from the eigenfunctions corresponding to the non-zero eigenvalues $\{\lambda_k\}_{k \geq 1}$ of the integral operator $K:L^{2}([a,b]) \rightarrow L^{2}([a,b])$,
\begin{align*}
(Kf)(s):=\int_a^b\cov(X_{s},X_t)f(t)\d t.
\end{align*} 
When the covariance satisfies a continuity condition it is known  (see for example \cite{Ghanesh} Section 2.3.3) that the normalized set of eigenfunctions $\{e_k\}_{k \geq 1}$ of $K$ is countable and forms a basis for $L^{2}([a,b])$. When we choose this basis in \eqref{eq:eigen} we adopt the special notation $\{Z_k\}_{d \geq 1}$ for the stochastic coefficients. In this case, the expansion is optimal in a number of ways. Specifically, we have:
\begin{theorem}[The Karhunen-Lo\`{e}ve Theorem]\label{theo:klt}
Let $X$ be a real valued continuous time stochastic process on $[a,b]$ such that $0 \leq a \leq b < \infty$ and let $\e[X_t] = 0$ and $\e[X^2_t] < \infty$ for each $t \in [a,b]$. Further, suppose $\cov(X_s,X_t)$ is continuous on $[a,b]\times[a,b]$. 
\begin{enumerate}[(i)]
\item Then,
\begin{align*}
\e\left[\left( X_t - \sum_{k=1}^{d}Z_ke_k(t)\right)^2\right] \rightarrow 0,\quad\text{ as }\quad d \rightarrow \infty
\end{align*}
uniformly for $t \in [a,b]$. Additionally, the random variables $\{Z_k\}_{k \geq 1}$ are uncorrelated and satisfy $\e[Z_k] = 0$ and $\e[Z_k^2] = \lambda_k$.
\item For any other basis $\{\phi_k\}_{k \geq 1}$ of $L^2([a,b])$, with corresponding stochastic coefficients $\{Y_k\}_{k \geq 1}$, and any $d \in \mathbb{N}$, we have
\begin{align*}
\int_{a}^{b}\e\left[\left(\varepsilon_d(t)\right)^2\right]\d t \leq \int_{a}^{b}\e\left[\left(\tilde\varepsilon_d(t)\right)^2\right]\d t,
\end{align*}
where $\varepsilon_d$ and $\tilde\varepsilon_d$ are the remainders $\varepsilon_d(t) :=\sum_{d+1}^{\infty}Z_ke_k(t)$ and $\tilde\varepsilon_d(t) :=\sum_{d+1}^{\infty}Y_k\phi_k(t)$.
\end{enumerate}
\end{theorem}
\noindent Going forward we assume the order of the eigenvalues, eigenfunctions, and the stochastic coefficients is determined according to $\lambda_1 \geq \lambda_2 \geq \lambda_3, \ldots$. \\ \\
\noindent According to Ghanem and Spanos \cite{Ghanesh} the Karhunen-Lo\`{e}ve theorem was proposed independently by Karhunen \cite{Karhun}, Lo\`{e}ve \cite{love}, and Kac and Siegert \cite{Katchy}. Modern proofs of the first part of the theorem can be found in \cite{Ashy} and \cite{Ghanesh} and the second part -- the optimality of the truncated approximation -- is also proven in \cite{Ghanesh}. A concise and readable overview of this theory is given in \cite[Chapter 7.1]{Luo}.\\ \\
\noindent We see that although the KLT is quite general, it is best applied in practice when can determine the three components necessary for a Karhunen-Lo\'{e}ve expansion: the eigenfunctions $\{e_k\}_{k \geq 1}$; the eigenvalues $\{\lambda_k\}_{k \geq 1}$; and the distribution of the stochastic coefficients $\{Z_k\}_{k \geq 1}$. If we wish to use the KLE for simulation then we need even more: We also need to know how to simulate the random vector $Z^{(d)} = (Z_1,Z_2,\ldots,Z_d)$ which, in general, has uncorrelated but not necessarily independent components. \\ \\
For Gaussian processes, the second obstacle is removed, since one can show that the $\{Z_k\}_{k \geq 1}$ are again Gaussian, and therefore independent. There are, of course, many ways to simulate a vector of independent Gaussian random variables. For a process $X \in \mathcal{K}$, the matter is slightly more complicated as we establish in Theorem \ref{theo:main1}. However, since the covariance function of a process $X \in \mathcal{K}$ differs from that of a Wiener process only by the scaling factor $\alpha$, the method for determining the eigenfunctions and the eigenvalues for a L\'{e}vy process is identical to that employed for a Wiener process. Therefore, we omit the proof of the following proposition, and direct the reader to \cite[pg. 41]{Ashy} where the proof for the Wiener process is given.
\begin{proposition}\label{prop:joe}
The eigenvalues and associated eigenfunctions of the operator $K$ defined on $L^{2}([0,T])$ with respect to $X \in \mathcal{K}$ are given by
\begin{align}\label{eq:efuncval}
\lambda_k= \frac{\alpha T^2}{\pi^2\left(k - \frac{1}{2}\right)^2},\quad\text{ and }\quad  e_k(t)= \sqrt{\frac{2}{T}}\sin\left(\frac{\pi}{T}\left(k-\frac{1}{2}\right)t\right),\quad k\in \mathbb{N},\,t\in[0,T].
\end{align}
\end{proposition}
\noindent A nice consequence of Proposition \ref{prop:joe} and Theorem \ref{theo:klt} is that it allows us to estimate the amount of total variance $v(T) := \int_0^T\var(X_t)\d t = \int_0^T\e[X^2_t]\d t = \alpha T^2/2$ we capture when we represent our process by a truncated KLE. Using the orthogonality of the $\{e_k\}_{k \geq 1}$, and the fact that $\e[Z_k^2] = \lambda_k$ for each $k$, it is straightforward to show that the total variance satisfies $v(T) = \sum_{k \geq 1}\lambda_k$. Therefore, the total variance explained by a $d$-term approximation is 
\begin{align*}
\frac{\sum_{k=1}^{d}\lambda_k}{v(T)} = \frac{2}{\pi^2}\sum_{k=1}^{d}\frac{1}{\left(k-\frac{1}{2}\right)^2}.
\end{align*}
By simply computing the quantity on the right we find that the first 2, 5 and 21 terms already explain $90\%$,  $95\%$, and $99\%$ of the total variance of the process. Additionally, we see that this estimate holds for all $X \in \mathcal{K}$ independently of $\alpha$ or $T$.\\ \\
\noindent The following lemma is the important first step in identifying the joint distribution of the stochastic coefficients of the KLE for $X \in \mathcal{K}$. The reader should note, however, that the lemma applies to more general L\'{e}vy processes, and is not just restricted to the set $\mathcal{K}$.
\begin{lemma}\label{lem:bert}
Let $X$ be a L\'evy process and let  $\{f_k\}_{k = 1}^d$ be a collection of functions which are in $L^1([0,T])$. Then the vector $\mathbf{\xi}$ consisting of elements
\begin{align*}
\xi_k = \int_0^{T}X_tf_k(s)\d s, \quad k \in \{1,2,\ldots, d\},
\end{align*}
has an ID distribution with characteristic exponent
\begin{align}\label{eq:bert}
\Psi_{\mathbf{\xi}}(\mathbf{z}) = \int_0^T\Psi_X\left(\langle \mathbf{z}, \mathbf{u}(t) \rangle\right)\d t,\quad \mathbf{z} \in \r^d,
\end{align}
where $\mathbf{u}:[0,T]\rightarrow \r^{d}$ is the function with $k$-th component $u_{k}(t) :=\int_t^T f_{k}(s)\d s$, $k \in \{1,2,\ldots,d\}$. 
\end{lemma}
\begin{remark}
A similar identity to \eqref{eq:bert} is known, see pg. 128 in \cite{handbook}. In the proof of Lemma \ref{lem:bert}, we borrow some ideas from there. Since the proof is rather lengthy we relegate it to the Appendix.
\end{remark}
\noindent With Lemma \ref{lem:bert} and Proposition \ref{prop:joe} in hand, we come to our first main result. In the following theorem we identify the generating triple of the vector $Z^{(d)}$ containing the first $d$ stochastic coefficients of the KLE for a process $X \in \mathcal{K}$. Although it follows that $Z^{(d)}$ has dependent entries (see Corollary \ref{cor:main2}), Theorem \ref{theo:main1}, and in particular the form of the L\'{e}vy measure $\Pi$, will also be the key to simulating $Z^{(d)}$. Going forward we use the notation $\mathcal{B}_{S}$ for the Borel sigma algebra on the topological space $S$.
\begin{theorem}\label{theo:main1}
If $X \in \mathcal{K}$ with generating triple $(0,\sigma^2,\nu)_{h \equiv 1}$ then $Z^{(d)} \in \mathcal{C}$ with generating triple \\ $(\mathbf{0},\mathcal{Q},\Pi)_{h\equiv 1}$ where
$\mathcal{Q}$ is a diagonal $d\times d$ matrix with entries
\begin{align}\label{eq:themat}
q_{k,k} := \frac{\sigma^2}{2}\frac{T^2}{\pi^2\left(k - \frac{1}{2}\right)^2},\quad k \in \{1,2,\ldots,d\},
\end{align}
and $\Pi$ is the measure,
\begin{align}\label{eq:levymeas}
\Pi(B) := \int_{\r\backslash\{0\}\times[0,T]}\ind(f(\mathbf{v}) \in B)(\nu\times\lambda)(\d \mathbf{v}),\quad B \in \mathbb{B}_{\rdz},
\end{align} 
where $\lambda$ is the Lebesgue measure on $[0,T]$ and $f:\r\times[0,T] \rightarrow \r^d$ is the function
\begin{align}\label{eq:theff}
(x,t) \mapsto \frac{\sqrt{2T}x}{\pi}\left(
\frac{\cos\left(\frac{\pi}{T}\left(1 - \frac{1}{2}\right) t \right)}{\left(1 - \frac{1}{2}\right)},\frac{\cos\left(\frac{\pi}{T}\left(2 - \frac{1}{2}\right) t \right)}{\left(2 - \frac{1}{2}\right)},\ldots,\frac{\cos\left(\frac{\pi}{T}\left(d - \frac{1}{2}\right) t \right)}{\left(d - \frac{1}{2}\right)}\right)^{\textnormal{\textbf{T}}}.
\end{align}
\end{theorem}
\begin{proof}
We substitute the formula for the characteristic exponent (Formula \ref{eq:char} with $a=0$ and $h\equiv 1$)  and the eigenfunctions (Formula \ref{eq:efuncval}) into \eqref{eq:bert} and carry out the integration. Then \eqref{eq:themat} follows from the fact that 
\begin{align*}
u_k(t) = \int_t^Te_k(s)\d s = \sqrt{\frac{2}{T}}\int_t^T\sin\left(\frac{\pi}{T}\left(k - \frac{1}{2}\right)s\right)\d s = \sqrt{2T}\frac{\cos\left(\frac{\pi}{T}\left(k - \frac{1}{2}\right)t\right)}{\pi(k-\frac{1}{2})}, \quad k \in \mathbb{N}
\end{align*} 
and that the $\{u_k\}_{k \geq 1}$ are therefore also orthogonal on $[0,T]$. \\ \\
\noindent Next we note that $f$ is a continuous function from $\r\backslash\{0\}\times[0,T]$ to $\r^d$ and is therefore $\left(\mathcal{B}_{\r\backslash\{0\}\times[0,T]},\mathcal{B}_{\r^d\backslash\{\mathbf{0}\}}\right)$ measurable. Therefore, $\Pi$ is nothing other than the push forward measure obtained from $(\nu\times\lambda)$ and $f$; in particular, it is a well-defined measure on $\mathcal{B}_{\r^d\backslash\{\mathbf{0}\}}$. It is also a L\'{e}vy measure that satisfies \ref{eq:conditionA} since
\begin{align}\label{eq:prolev}
\int_{\vert \mathbf{x} \vert > 1}\vert \mathbf{x} \vert^2\Pi(d \mathbf{x}) \leq \int_{\r^{d}\backslash\{\mathbf{0}\}}\vert \mathbf{x}\vert^2\Pi(\d \mathbf{x})
= \frac{2T}{\pi^2}\int_0^T\left(\sum_{k=1}^du_k^2(t)\right)\d t\int_{\r\backslash\{0\}}x^2\nu(d x) < \infty,
\end{align}
where the final inequality follows from the fact that $X \in \mathcal{K}$. Applying Fubini's theorem and a change of variables, i.e.,
\begin{align*}
\int_0^T\int_{\r\backslash\{0\}}e^{\i x\langle \mathbf{z}, \mathbf{u}(t) \rangle} - 1 - \i x\langle \mathbf{z},\mathbf{u}(t) \rangle \nu(\d x)\d t &= \int_{\r\backslash\{0\}\times[0,T]}e^{\i \langle \mathbf{z}, f(\mathbf{v}) \rangle} - 1 - \i \langle \mathbf{z},f(\mathbf{\mathbf{v}}) \rangle (\nu\times\lambda)(\d \mathbf{v})\\  &= \int_{\r^d\backslash\{\mathbf{0}\}}e^{\i\langle \mathbf{z}, \mathbf{x} \rangle} - 1 - \i\langle \mathbf{z},\mathbf{x} \rangle \Pi(\d \mathbf{x}),
\end{align*}
 concludes the proof of infinite divisibility. Finally, noting that
 \begin{align*}
 \e[Z_k] = \e\left[\int_0^TX_te_k(t)\d t\right] = \int_0^T\e[X_t]e_k(t)\d t = 0,\quad k \in \{1,2,\ldots,d\},
 \end{align*}
 shows that $Z^{(d)} \in \mathcal{C}$.
\end{proof}
\begin{remark}
Note, that if we set $\sigma=1$, $\nu \equiv 0$, and $T=1$ we may easily recover the KLE of the Wiener process, i.e., \eqref{eq:wklt}, from Theorem \ref{theo:main1}.
\end{remark}
\noindent We gather some fairly obvious but important consequences of Theorem \ref{theo:main1} in the following corollary.
\begin{corollary}\label{cor:main1} Suppose $X \in \mathcal{K}$, then:
\begin{enumerate}[(i)]
\item $X \in \mathcal{KB}$ with generating triple $(a,\sigma^2,\nu)_{h\equiv 0}$ if, and only if, $Z^{(d)} \in \mathcal{CB}$ with generating triple $(\mathbf{a},\mathcal{Q},\Pi)_{h\equiv 0}$, where $\mathcal{Q}$ and $\Pi$ are as defined in \eqref{eq:themat} and \eqref{eq:levymeas} and $\mathbf{a}$ is the vector with entries
\begin{align}\label{eq:drift}
a_k := a\frac{(-1)^{k+1}\sqrt{2}T^{\frac{3}{2}}}{\pi^2\left(k - \frac{1}{2}\right)^2},\quad k \in \{1,2,\ldots,d\}.
\end{align}
\item $X$ has finite L\'{e}vy measure $\nu$ if, and only if, $Z^{(d)}$ has finite L\'{e}vy measure $\Pi$.
\end{enumerate}
\end{corollary}
\begin{proof}\ \\
\emph{(i)} Since 
\begin{align*}
\int_{\rdz}\vert \mathbf{x} \vert \Pi(\d \mathbf{x}) = \frac{\sqrt{2T}}{\pi}\int_{0}^{T}\left\vert\sum_{k=1}^{d}u_k(t)\right\vert\d t\int_{\rz}\vert x \vert \nu(\d x)
\end{align*}
and \ref{eq:conditionA} is satisfied by both $\nu$ and $\Pi$ it follows that \ref{eq:conditionB} is satisfied for $\nu$ if, and only if, it is satisfied for $\Pi$. Formula \ref{eq:drift} then follows from the fact that
\begin{align*}
-\int_{\rdz}P_k(\bfx) \Pi(\d \bfx) = -\int_{\rz}x\nu(\d x)\frac{\sqrt{2T}}{\pi}\int_0^T\frac{\cos\left(\frac{\pi}{T}\left(k - \frac{1}{2}\right) t \right)}{\left(k - \frac{1}{2}\right)}\d t = a\frac{(-1)^{k+1}\sqrt{2}T^{\frac{3}{2}}}{\pi^2\left(k - \frac{1}{2}\right)^2}.
\end{align*}
\noindent \emph{(ii)} Straightforward from the definition of $\Pi$ in Theorem \ref{theo:main1}.
\end{proof}
\noindent Also intuitively obvious, but slightly more difficult to establish rigorously, is the fact that the entries of $Z^{(d)}$ are dependent unless $\nu \equiv 0$.
\begin{corollary}\label{cor:main2}
If $X \in \mathcal{K}$ then $Z^{(d)}$ has independent entries if, and only if, $\nu$ is the zero measure.
\end{corollary}
\noindent To prove Corollary \ref{cor:main2} we use the fact that a $d$-dimensional ID random vector with generating triple $(\mathbf{a},Q,\nu)$ has independent entries if, and only if, $\nu$ is supported on the union of the coordinate axes and $Q$ is diagonal (see E 12.10 on page 67 in \cite{sato}). For this purpose we define, for a vector $\mathbf{x} = (x_1,x_2,\ldots,x_d)^{\textbf{T}} \in \r^d$ such that $\,x_k > 0,\,k\in\{1,2,\ldots,d\}$, the sets
\begin{align*}
\mathcal{I^+}(\mathbf{x}) := \Pi_{k=1}^{d}(x_k,\infty),\quad\text{and}\quad\mathcal{I^-}(\mathbf{x}) := \Pi_{k=1}^{d}(-\infty,-x_k),
\end{align*}
where we caution the reader that the symbol $\Pi$ indicates the Cartesian product and not the L\'{e}vy measure of $Z^{(d)}$. \\ \\
\noindent In the proof below, and throughout the remainder of the paper, $f$ will always refer to the function defined in \eqref{eq:theff}, and $f_k$ to the $k$-th coordinate of $f$.

\begin{proof}[Proof of Corollary \ref{cor:main2}]\ \\
\noindent ($\Leftarrow$) The assumption $\nu \equiv 0$ implies our process is a scaled Wiener process in which case it is well established that $Z^{(d)}$ has independent entries. Alternatively, this follows directly the fact that the matrix $\mathcal{Q}$ in Theorem \ref{theo:main1} is diagonal.\\ \\
\noindent ($\Rightarrow$) We assume that $\nu$ is not identically zero and show that there exists $\mathbf{x}$ such that either $\Pi(\mathcal{I}^+(\mathbf{x}))$ or $\Pi(\mathcal{I}^-(\mathbf{x}))$ is strictly greater than zero.\\ \\
\noindent Since $\nu(\rz) > 0$ there must exist $\delta > 0$ such that one of $\nu((-\infty,-\delta))$ and $\nu((\delta,\infty))$ is strictly greater than zero; we will initially assume the latter. We observe that for $d \in \mathbb{N}$, $d \geq 2$, the zeros of the function $h_d:[0,T]\rightarrow\r$ defined by
\begin{align*}
t \mapsto \frac{\cos\left(\frac{\pi}{T}\left(d - \frac{1}{2}\right) t \right)}{\left(d - \frac{1}{2}\right)},
\end{align*}
occur at points $\{nT/(2d-1)\}_{n=1}^{2d -1}$, and therefore the smallest zero is $t_d := T/(2d-1)$. From the fact that the cosine function is positive and decreasing on $[0,\pi/2]$ we may conclude that
\begin{align*}
\frac{\cos\left(\frac{\pi}{T}\left(k - \frac{1}{2}\right) t \right)}{\left(k - \frac{1}{2}\right)} > \epsilon,\quad k\in\{1,2,\ldots,d\},\quad t \in [0,t_d/2],
\end{align*}
where $\epsilon = h_d(t_d/2) > 0$. Now, let $\mathbf{x}$ be the vector with entries $x_k = \delta\epsilon\sqrt{2T}/\pi$ for $k \in \{1,2,\ldots,d\}$. Then,
\begin{align*}
(\delta,\infty) \times [0,t_d/2] \subset f^{-1}\left(\mathcal{I}^+\left(\mathbf{x}\right)\right),
\end{align*}
since for $(x,t)\in (\delta,\infty) \times [0,t_d/2]$ we have 
\begin{align*}
f_k(x,t) = \frac{\sqrt{2T}}{\pi}x\frac{\cos\left(\frac{\pi}{T}\left(k - \frac{1}{2}\right) t \right)}{\left(k - \frac{1}{2}\right)} > \delta\epsilon\frac{\sqrt{2T}}{\pi} = x_k,\quad k \in \{1,2,\ldots,d\}.
\end{align*}
But then,
\begin{align}
\Pi(\mathcal{I}^+(\mathbf{x})) \geq \nu((\delta,\infty))\times\lambda([0,t_d/2]) > 0.
\end{align}
If we had initially assumed that $\nu((-\infty,-\delta)) > 0$ we would have reached the same conclusion by using the interval $(-\infty,-\delta)$ and $\mathcal{I}^-(\mathbf{x})$. We conclude that $\Pi$ is not supported on the union of the coordinate axes, and so $Z^{(d)}$ does not have independent entries.
\end{proof}
\section{Shot-noise representation of $Z^{(d)}$}\label{sec:shotnoise}
Although we have characterized the distribution of our stochastic coefficients $Z^{(d)}$ we are faced with the problem of simulating a random vector with dependent entries with only the knowledge of the characteristic function. In general, this seems to be a difficult problem, even generating random variables from the characteristic function is not straightforward (see for example \cite{fromchar}). In our case, thanks to Theorem \ref{theo:main1} we know that $Z^{(d)}$ is infinitely divisible and that the L\'{e}vy measure $\Pi$ has a special disintegrated form. This will help us build the connection with the so-called shot-noise representation of our vector $Z^{(d)}$. The goal is to represent $Z^{(d)}$ as an almost surely convergent series of random vectors. \\ \\
\noindent To explain this theory -- nicely developed and explained in \cite{rosy,theoapp} -- we assume that we have two random sequences $\{V_i\}_{i \geq 1}$ and $\{\Gamma_i\}_{i \geq 1}$ which are independent of each other and defined on a common probability space. We assume that each ${\Gamma_i}$ is distributed like a sum of $i$ independent exponential random variables with mean 1, and that the $\{V_i\}_{i \geq 1}$ take values in a measurable space $D$, and are i.i.d. with  common distribution $F$. Further, we assume we have a measurable function $H:(0,\infty)\times D \rightarrow \r^{d}$ which we use to define the random sum
\begin{align}\label{eq:shot}
S_n:= \sum_{i = 1}^{n}H(\Gamma_i,V_i),\quad n \in \mathbb{N},
\end{align}
and the measure
\begin{align}\label{eq:meas}
\mu(B) := \int_0^{\infty}\int_{D}\ind(H(r,v) \in B)F(\d v)\d r,\quad B \in B_{\rdz}.
\end{align}
The function $C:(0,\infty)\rightarrow \r^{d}$ is defined by
\begin{align}\label{eq:funA}
C_k(s) := \int_0^{s}\int_DP_k(H(r,v))F(\d v)\d r,\quad k \in \{1,2,\ldots,d\},
\end{align}
where, as before, $P_k$ is the projection onto the $k$-th component.
The connection between \eqref{eq:shot} and ID random vectors is then explained in the following theorem whose results can be obtained by restricting Theorems 3.1, 3.2, and 3.4 in \cite{rosy} from a general Banach space setting to $\r^d$.
\begin{theorem}[Theorems 3.1, 3.2, and 3.4 in \cite{rosy}]\label{theo:ros}
Suppose $\mu$ is a L\'evy measure, then:
\begin{enumerate}[(i)]
\item If \ref{eq:conditionB} holds then $S_n$ converges almost surely to an ID random vector with generating triple $(\mathbf{0},\mathbf{0},\mu)_{h\equiv 0}$ as $n \rightarrow \infty$.
\item If \ref{eq:conditionA} holds, and for each $v \in S$ the function $r \rightarrow \vert H(r,v) \vert$ is non increasing, then 
\begin{align}\label{eq:compsum}
M_n := S_n - C(n), \quad n\in \mathbb{N}
\end{align}
converges almost surely to an ID random vector with generating triple $(\mathbf{0},\mathbf{0},\mu)_{h\equiv 1}$.
\end{enumerate}

\end{theorem}

\noindent The name ``shot-noise representation" comes from the idea that $\vert H \vert$ can be interpreted as a model for the volume of the noise of a shot $V_i$ that occurred $\Gamma_i$ seconds ago. If $\vert H \vert $ is non increasing in the first variable, as we assume in case (ii) in Theorem \ref{theo:ros}, then the volume decreases as the elapsed time grows. The series $\lim_{n\rightarrow \infty}S_n$ can be interpreted as the total noise at the present time of all previous shots.\\ \\
\noindent The goal is to show that for any process in $\mathcal{K}$ whose L\'{e}vy measure admits a strictly positive density $\pi$, the vector $Z^{(d)}$ has a shot-noise representation of the form \eqref{eq:shot} or \eqref{eq:compsum}. To simplify notation we make some elementary but necessary observations/assumptions: First, we assume that $X$ has no Gaussian component $\sigma^2$. There is no loss of generality to this assumption, since if $X$ does have a Gaussian component then $Z^{(d)}$ changes by the addition of a vector of independent Gaussian random variables. This poses no issue from a simulation standpoint. Second, from \eqref{eq:char} we see that any L\'{e}vy process $X$ with representation $(0,0,\pi)_{h\equiv j}$, $j \in \{0,1\}$ can be decomposed into the difference of two independent L\'{e}vy processes, each having only positive jumps. Indeed, splitting the integral and making a change of variable $x \mapsto -x$ gives
\begin{align}\label{eq:split}
\Psi_{X}(z) &= -\int_{\rz} e^{\i zx} - 1 - \i zxj\pi(x)\d x \nonumber \\
	&= -\int_{0}^{\infty} e^{\i zx} - 1 - \i zxj\pi(x)\d x -\int_{0}^{\infty}e^{\i z(-x)} - 1 - \i z(-x)j\pi(-x)\d x \nonumber \\
	&= \Psi_{X^+}(z) + \Psi_{-X^-}(z)
\end{align}
where $X^+$ (resp. $X^-$) has L\'{e}vy density $\pi(\cdot)$ (resp. $\pi(-\cdot)$) restricted to $(0,\infty)$. In light of this observation, the results of Theorem \ref{theo:main2} are limited to L\'{e}vy processes with positive jumps. It should be understood that for a general process we can obtain $Z^{(d)}$ by simulating $Z^{(d)}_+$ and $Z^{(d)}_-$ -- corresponding to $X^+$ and $X^-$ respectively -- and then subtracting the second from the first to obtain a realization of $Z^{(d)}$.\\ \\
\noindent Last, for a L\'{e}vy process with positive jumps and strictly positive L\'{e}vy density $\pi$, we define the function
\begin{align}\label{eq:thefung}
g(x) := \int_x^{\infty}\pi(s)\d s.
\end{align}
which is just the tail integral of the L\'{e}vy measure. We see that $g$ is strictly monotonically decreasing to zero, and so admits a strictly monotonically decreasing inverse $g^{-1}$ on the domain $(0,g(0))$.
\begin{theorem}\label{theo:main2} Let $\pi$ be a strictly positive L\'{e}vy density on $(0,\infty)$ and identically zero elsewhere.
\begin{enumerate}[(i)]
\item If $X \in \mathcal{KB}$ with generating triple $(a,0,\pi)_{h\equiv 0}$, then $Z^{(d)}$ has a shot noise representation
\begin{align}\label{eq:shotnoise}
Z^{(d)} \edist \mathbf{a} + \sum_{i \geq 1}H(\Gamma_i,U_i)
\end{align}
where $f$ and $\mathbf{a}$ are defined in \eqref{eq:theff} and \eqref{eq:drift} respectively, $\{U_i\}_{i \geq 1}$ is an i.i.d. sequence of uniform random variables on $[0,1]$, and
\begin{align}\label{eq:H}
H(r,v) := f(g^{-1}(r/T)\ind( 0 < r < g(0)),Tv).
\end{align}
\item If $X \in \mathcal{K}$ with generating triple $(0,0,\pi)_{h\equiv 1}$, then $Z^{(d)}$ has a shot noise representation
\begin{align}\label{eq:const1}
Z^{(d)} \edist \lim_{n\rightarrow\infty}\sum_{i = 1}^{n}H(\Gamma_i,U_i) - C(n),
\end{align}
where $H$ and $\{U_i\}_{i\geq 1}$ are as in Part $(i)$ and $C$ is defined as in \eqref{eq:funA}.
\end{enumerate}
\end{theorem}
\begin{proof}
Rewriting \eqref{eq:levymeas} to suit our assumptions and making a change of variables $t = Tv$ gives, for any $B \in \mathcal{B}_{\rdz}$
\begin{align*}
\Pi(B) &= \int_0^T\int_0^{\infty}\ind(f(x,t) \in B)\pi(x)\d x\d t = \int_0^1\int_0^{\infty}\ind(f(x,Tv) \in B)T\pi(x)\d x\d v.
\end{align*}
Making a further change of variables $r = Tg(x)$ gives
\begin{align*}
\Pi(B) &= \int_0^1\int_0^{g(0)}\ind(f(g^{-1}(r/T),Tv) \in B)\d r\d v.
\end{align*}
Since $0 \notin B$, so that $\ind(0 \in B) = 0$, we may conclude that
\begin{align*}
\Pi(B) = \int_0^{\infty}\int_{0}^1\ind(f(g^{-1}(r/T)\ind( 0 < r < g(0)),Tv) \in B)\d v\d r.
\end{align*}
From the definition of the function $f$ (Formula \ref{eq:theff}), and that of $g^{-1}$, it is clear that 
\begin{align}\label{eq:myh}
(r,v) \mapsto f(g^{-1}(r/T)\ind( 0 < r < g(0)),Tv)
\end{align}
is measurable and non increasing in absolute value for any fixed $v$. Therefore, we can identify \eqref{eq:myh} with the function $H$, the uniform distribution on $[0,1]$ with $F$, and $\Pi$ with $\mu$. The results then follow by applying the results of Theorems \ref{theo:main1} and \ref{theo:ros} and Corollary \ref{cor:main1}.
\end{proof}
\noindent Going forward we will write simply $H(r,v) = f(g^{-1}(r/T),Tv)$ where it is understood that $g^{-1}$ vanishes outside the interval $(0,g(0))$.
\subsubsection*{Discussion}
There are two fairly obvious difficulties with the series representations of Theorem \ref{theo:main2}. The first -- this a common problem for all series representations of ID random variables when the L\'evy measure is not finite -- is that we have to truncate the series when $g(0) =\infty$ (equivalently $\nu(\rz) = \infty$). Besides the fact that in these cases our method fails to be exact, computation time may become an issue if the series converge too slowly. The second issue is that $g^{-1}$ is generally not known in closed form. Thus, in order to apply the method we will need a function $g$ that is amenable to accurate and fast numerical inversion. In the survey \cite{rosy} Rosi{\'n}ski reviews several methods, which depend on various properties of the L\'evy measure (for example, absolute continuity with respect to a probability distribution), that avoid this inversion. In a subsequent paper \cite{rosinski2007tempering} he develops special methods for the family of tempered $\alpha$-stable distributions that also do not require inversion of the tail of the L\'evy measure. We have made no attempt to adapt these techniques here, as the fall outside of the scope of this paper. However, this seems to be a promising area for further research.\\ \\
A nice feature of simulating a $d$-dimensional KLE of a L\'{e}vy process $X \in \mathcal{K}$ via Theorem \ref{theo:main2} is that we may increase the dimension incrementally. That is, having simulated a path of the $d$-term KLE approximation of $X$, 
\begin{align}\label{eq:partsum}
S_{t}^{(d)} := \sum_{k=1}^{d}Z_ke_k(t),\quad t\in[0,T],
\end{align}
we may derive a path of $S^{(d+1)}$ directly from $S^{(d)}$ as opposed to starting a fresh simulation. We observe that a realization $z_k$ of $Z_k$ can be computed individually once we have the realizations $\{\gamma_i,u_i\}_{i\geq 1}$ of $\{\Gamma_i,\,U_i\}_{i\geq 1}$. Specifically,
\begin{align*}
z_k = a_k + \sum_{i \geq 1}\frac{\sqrt{2T}g^{-1}(\gamma_i/T)}{\pi}\frac{\cos\left(\pi\left(k - \frac{1}{2}\right)u_i\right)}{\left(k-\frac{1}{2}\right)},
\end{align*}
when \ref{eq:conditionB} holds, with an analogous expression when it does not. Thus, if $s^{(d)}_t$ is our realization of $S^{(d)}_t$ we get a realization of $S^{(d+1)}_t$ via $s^{(d+1)}_t = s^{(d)}_t + z_{d+1}e_{d+1}(t)$.\\ \\
\noindent It is also worthwhile to compare the series representations for L\'{e}vy processes found in \cite{rosy} and the proposed method. As an example, suppose we have a subordinator $X$ with a strictly positive L\'{e}vy density $\pi$. Then, it is also true that
\begin{align}\label{eq:regseries}
\{X_t:t\in[0,T]\} \edist \left\{\sum_{i\geq 1}g^{-1}(\Gamma_i/T)\ind(TU_i < t):t\in[0,T]\right\}.
\end{align}
The key difference between the approaches, is that the series in \eqref{eq:regseries} depends on $t$, whereas the series representation of $Z^{(d)}$ is independent of $t$. Therefore, in \eqref{eq:regseries} we have to recalculate the series for each $t$, adding those summands for which $U_iT < t$. Of course, the random variables $\{\Gamma_i, U_i\}_{i\geq 1}$ need to be generated only once. On the other hand, while we have to simulate $Z^{(d)}$ only once for all $t$, each summand requires the evaluation of $d$ cosine functions, and for each $t$ we have to evaluate $d$ sine functions when we form the KLE. However, since there is no more randomness once we have generated $Z^{(d)}$ the second computation can be done in advance.

\subsubsection*{Example} Consider the Variance Gamma (VG) process which was first introduced in \cite{initvg} and has since become a popular model in finance. The process can be constructed as the difference of two independent Gamma processes, i.e., processes with L\'{e}vy measures of the form
\begin{align}\label{eq:gammasub}
\nu(\d x) = c\frac{e^{-\rho x}}{x}\d x,\quad x > 0,
\end{align}
where $c,\,\rho > 0$. For this example we use a Gamma process $X^+$ with parameters $c=1$ and $\rho=1$ and subtract a Gamma process $X^-$ with parameters $c=1$ and $\rho = 2$ to yield a VG process $X$. Assuming no Gaussian component or additional linear drift, it can be shown (see Proposition 4.2 in \cite{contan}) that the characteristic exponent of $X$ is then
\begin{align*}
\Psi_X(z) = -\left(\int_{0}^{\infty}(e^{\i zx} -1)\frac{e^{-x}}{x}\d x + \int_0^{\infty}(e^{-\i zx} -1)\frac{e^{-2x}}{x}\d x \right) = \log\left(1-\i z\right) + \log\left(1 + \frac{\i z}{2}\right).
\end{align*}
We observe that $X^+,\,X^{-} \notin \mathcal{K}$ since
\begin{align*}
\e[X^+_t] = \i t\Psi'_{X^+}(0) = t \neq 0\quad\text{ and }\quad\e[X^-_t] = \i t\Psi'_{X^-}(0) = \frac{t}{2} \neq 0.
\end{align*}
However, this is not a problem, since we can always construct processes $\tilde{X}^+,\,\tilde{X}^- \in \mathcal{K}$ by subtracting $t$ and $t/2$ from $X^+$ and $X^-$ respectively. We then generate the KLE of $\tilde{X}^+$ and add back $t$ to the result, and apply the analogous procedure for $X^{-}$. This is true generally as well, i.e., for a square integrable L\'{e}vy process with expectation $\e[X_t] = \i t\Psi_X'(0) \neq 0$ we can always construct a process $\tilde{X} \in \mathcal{K}$ by simply subtracting the expectation $\i t\Psi_X'(0)$.\\ \\
\noindent From \eqref{eq:gammasub} we see that the function $g$ will have the form
\begin{align*}
g(x) = c\int_x^{\infty}\frac{e^{-\rho s}}{s}\d s = cE_1(\rho x), 
\end{align*}
where $E_1(x) := \int_x^{\infty}s^{-1}e^{-s}\d s$ is the exponential integral function. Therefore, 
\begin{align*}
g^{-1}(T^{-1}r) = \frac{1}{\rho}E_1^{-1}\left(\frac{r}{Tc}\right).
\end{align*}
There are many routines available to compute $E_1$; we choose a Fortran implementation to create a lookup table for $E^{-1}_1$ with domain $[6.226\times10^{-22},45.47]$. We discretize this domain into $200000$ unevenly spaced points, such that the distance between two adjacent points is no more than 0.00231. Then we use polynomial interpolation between points.\\ \\
\noindent When simulating $Z^{(d)}_+$ we truncate the series \eqref{eq:shotnoise} when $(Tc)^{-1}\Gamma_i > 45.47$; at this point we have $g^{-1}(T^{-1}\Gamma_i) < \rho^{-1}10^{-19}$. Using the fact that the $\{\Gamma_i\}_{i\geq 1}$ are nothing other than the waiting times of a Poisson process with intensity one, we estimate that we need to generate on average $45Tc$ random variables to simulate $Z^{(d)}_+$ and similarly for $Z^{(d)}_-$. We remark that for the chosen process both the decay and computation of $g^{-1}$ are manageable.\\ \\
\noindent  We simulate sample paths of $S^{(d)}$ for $d \in \{5,10,15,20,25,100,3000\}$ using the described approach. We also compute a Monte Carlo (MC) approximation of the expectation of $X$ by averaging over $10^6$ sample paths of the $d$-term approximation. Some sample paths and the results of the MC simulation are depicted in Figure \ref{fig1}, where the colors black, grey, red, green, blue, cyan, and magenta correspond to $d$ equal to 5, 10, 15, 20, 25, 100, and 3000 respectively.\\ \\
\noindent In Figure \ref{figa} we show the sample paths resulting from a simulation of $S^{(d)}$. We notice that the numerical results correspond with the discussion of Section \ref{sec:klt}: the large movements of the sample path are already captured by the 5-term approximation. We also notice peaks resulting from rapid oscillations before the bigger ``jumps" in the higher term approximations. This behaviour is magnified for the 3000-term approximation in Figure \ref{figb}. In classical Fourier analysis this is referred to as the Gibbs phenomenon; the solution in that setting is to replace the partial sums by Ces\`{a}ro sums. We can employ the same technique here, replacing $S^{(d)}$ with $C^{(d)}$, which is defined by
\begin{align*}
C_t^{(d)} :=\frac{1}{d}\sum_{k=1}^dS_t^{(k)}.
\end{align*}
It is relatively straightforward to show that $C^{(d)}$ converges to $X$ in the same manner as $S^{(d)}$ (as described in Theorem \ref{theo:klt} (i)). In Figure \ref{figc} we show the effect of replacing $S^{(d)}$ with $C^{(d)}$ on all sample paths, and in Figure \ref{figd} we show the $C^{(3000)}$ approximation -- now the Gibbs phenomenon is no longer apparent. \\ \\
In Figure \ref{fige} we show the MC simulation of $E[S^{(5)}_t]$ (black +) plotted together with $E[X_t] = t/2$ (green $\circ$). We see the 5-term KLE already gives a very good approximation. In Figure \ref{figf} we also show the errors $E[S^{(d)}_t] - E[X_t]$ for $d=5$ (black +), $d=25$ (blue $\circ$), and $d=3000$ (magenta $\square$). Again we have agreement with the discussion in Section \ref{sec:klt}: little is gained in our MC approximation of $E[X_t]$ by choosing a KLE with more than 25 terms. Recall that a KLE with 25 terms already captures more than 99\% of the total variance of the given process.

\begin{figure}
\centering
\subfloat[]
{\label{figa}\includegraphics[height =5.5cm]{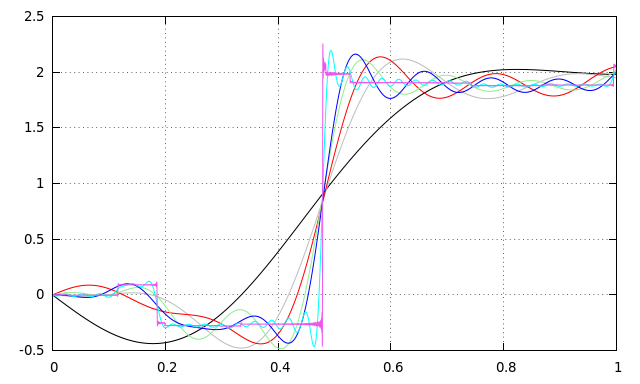}}
\subfloat[]
{\label{figb}\includegraphics[height =5.5cm]{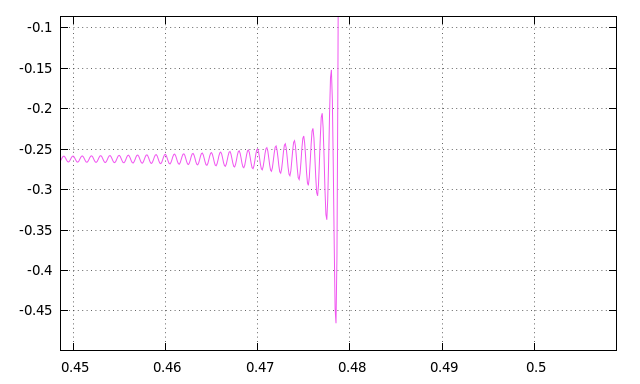}}\\
\subfloat[]
{\label{figc}\includegraphics[height =5.5cm]{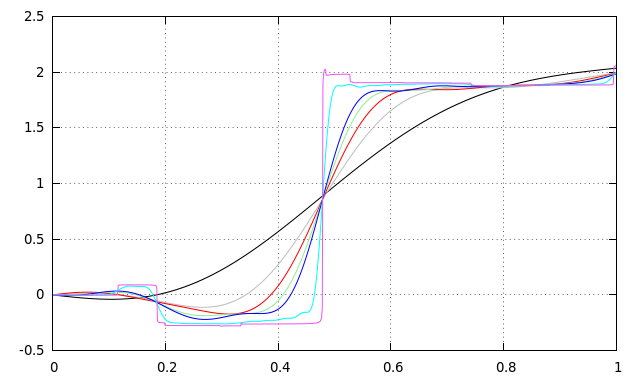}}
\subfloat[]
{\label{figd}\includegraphics[height =5.5cm]{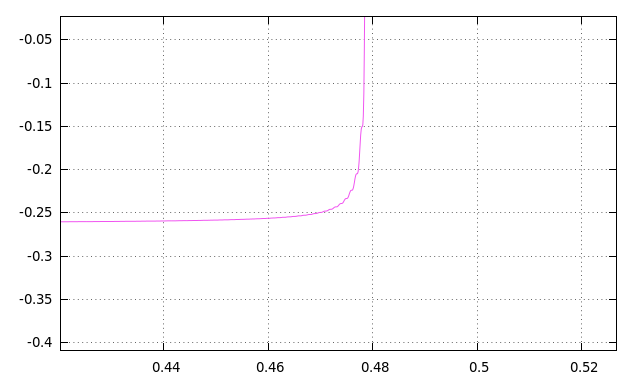}}\\
\subfloat[]
{\label{fige}\includegraphics[height =5.5cm]{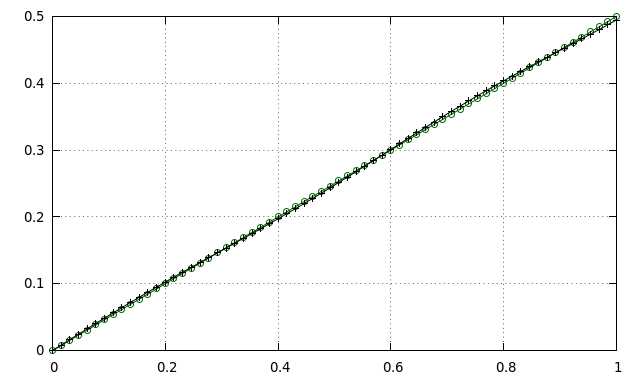}}
\subfloat[]
{\label{figf}\includegraphics[height =5.4cm]{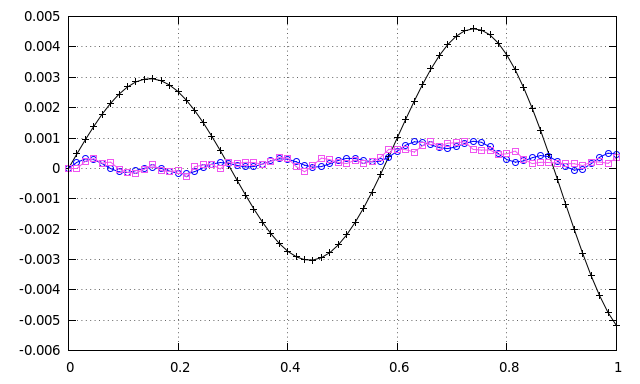}}   
\caption{(a) KLE sample paths  (b) Example of Gibbs phenomenon (c) KLE with Ces\`{a}ro sums \\(d) Mitigated Gibbs phen. (e) $\e[X_t]=t/2$ and MC sim. of $\e[S^{(5)}_t]$ (f) MC Err. $\e[S^{(d)}_t] - t/2$}
\label{fig1}
\end{figure}

\section*{Author's acknowledgements} My work is supported by the Austrian Science Fund (FWF) under the project F5508-N26, which is part of the Special Research Program ``Quasi-Monte Carlo Methods: Theory and Applications". I would like to thank Jean Bertoin for explaining his results in \cite{handbook} to me. This helped me extend identity \eqref{eq:bert} of Lemma \ref{lem:bert} from $C^1$ functions to $L^1$ functions. Further I would like to thank Alexey Kuznetsov and Gunther Leobacher for reading a draft of this paper and offering helpful suggestions.

\newpage
\begin{appendices}

\section{Additional proof}\label{app:A}
\begin{proof}[Proof of Lemma \ref{lem:bert}]
We give a proof for continuously differentiable $\{f_k\}_{k = 1}^d$ first and then prove the general case. Accordingly, we fix $\mathbf{z} \in \r^d$, a collection of continuously differentiable $\{f_k\}_{k = 1}^d$ defined on $[0,T]$, and a L\'{e}vy process $X$ with state space $\r$. Instead of proving identity \eqref{eq:bert} directly for $X$ we will prove that
\begin{align}\label{eq:withb}
\Psi_{\xi^{(b)}}(\mathbf{z}) = \int_0^T\Psi_{X^{(b)}}\left(\langle \mathbf{z}, \mathbf{u}(t) \rangle\right)\d t = b\int_0^T\Psi_X\left(\langle \mathbf{z}, \mathbf{u}(t) \rangle\right)\d t,\quad \mathbf{z} \in \r^d,\quad b > 0,
\end{align}
where $X^{(b)}$ is the process defined by $X^{(b)}_t := X_{bt}$ and $\xi^{(b)}$ is the vector with entries
\begin{align*}
\xi^{(b)}_k := \int_{0}^{T}X_{bt}f_k(t)\d t,\quad k \in \{1,2,\ldots, k\}.
\end{align*}
It is clear that $X^{(b)}$ is a L\'{e}vy process, that $\Psi_{X^{(b)}} = b\Psi_X$, and that \eqref{eq:bert} corresponds to the special case $b=1$. We focus on this more general result because it will lead directly to a proof of infinite divisibility. We begin by defining
\begin{align}\label{eq:arr}
{R^{(k)}_N} := \frac{T}{N}\sum_{n=0}^{N-1}f_k\left(\frac{(n+1)T}{N}\right)X_{\frac{b(n+1)T}{N}},\quad k \in \{1,2,\ldots, k\},\, N \in \mathbb{N},
\end{align}
which are $N$-point, right-endpoint Riemann sum approximations of the random variables $\xi^{(b)}_k$. By the usual telescoping sum technique for L\'evy processes we can write
\begin{align*}
X_{\frac{b(n+1)T}{N}} &= \left(X_{\frac{b(n+1)T}{N}} - X_{\frac{bnT}{N}}\right) + \left(X_{\frac{bnT}{N}} - X_{\frac{b(n-1)T}{N}}\right) + \ldots + \left(X_{\frac{b2T}{N}} - X_{\frac{bT}{N}}\right) + X_{\frac{bT}{N}} \\
&\edist X^{(1)} + X^{(2)} + \ldots +X^{(n+1)},
\end{align*}
where the random variables $X^{(i)}$ are independent and each distributed like $X_{bT/N}$. This allows us to rearrange the sum $R^{(k)}_N$ according to the random variables $X^{(i)}$, gathering together those with the same index. Therefore, we have
\begin{align*}
R^{(k)}_N \edist \sum_{n=0}^{N-1}X^{(n+1)}\left(\frac{T}{N}\sum_{j=n}^{N-1}f_k\left(\frac{(j+1)T}{N}\right)\right).
\end{align*}
We notice that the term in brackets on the right-hand side is a $(N - n)$-point, right-endpoint Riemann sum approximation for the integral of $f_k$ over the interval $[nT/N,T]$. Let us therefore define
\begin{align}\label{eq:esstee}
t^{(k)}_{n,N} := \frac{T}{N}\sum_{j=n}^{N-1}f_k\left(\frac{(j+1)T}{N}\right), \quad \text{ and } \quad s^{(k)}_{n,N} := \int_{\frac{nT}{N}}^{T}f_k(s)\d s,
\end{align}
as well as the $d$-dimensional vectors $\mathbf{t}_{n,N}$ and $\mathbf{s}_{n,N}$ consisting of entries $t^{(k)}_{n,N}$ and $s^{(k)}_{n,N}$ respectively. We observe that
\begin{align}\label{eq:convlev}
\e[\exp(\i\langle \mathbf{z}, \xi^{(b)}\rangle] =\lim_{N \rightarrow \infty}\e\left[\exp\left(\sum_{n=0}^{N-1}\i X^{(n+1)}\langle \mathbf{z}, \mathbf{t}_{n,N}\rangle\right)\right]
		   =\lim_{N \rightarrow \infty}\exp\left(-\frac{bT}{N}\sum_{n=0}^{N-1}\Psi_X(\langle \mathbf{z}, \mathbf{t}_{n,\,N}) \rangle\right),
\end{align}
where we have used the dominated convergence theorem to obtain the first equality, and the independence of the $X^{(i)}$ to obtain the final equality. Further, we get
\begin{align*}
\exp\left(-\int_0^T\Psi_{X^{(b)}}\left(\langle \mathbf{z}, \mathbf{u}(t) \rangle \right)\d t\right) = \lim_{N \rightarrow \infty}\exp\left(-\frac{bT}{N}\sum_{n=0}^{N-1}\Psi_X(\langle \mathbf{z}, \mathbf{s}_{n,\,N}) \rangle \right),
\end{align*}
by using the left-endpoint Riemann sums. We note that $\vert \langle \mathbf{z}, \mathbf{t}_{n,\,N} \rangle - \langle \mathbf{z}, \mathbf{s}_{n,\,N} \rangle \vert \rightarrow 0$ uniformly in $n$ since
\begin{align}\label{eq:bnd}
\vert \langle \mathbf{z}, \mathbf{t}_{n,\,N} \rangle -  \langle \mathbf{z}, \mathbf{s}_{n,\,N} \rangle\vert \leq \sum_{k=1}^{d}\vert z_k \vert\left \vert t^{(k)}_{n,N} - s^{(k)}_{n,N} \right \vert \leq \frac{d T^2}{N} \max_{1\leq k \leq d}\left\{\vert z_k\vert\sup_{x \in [0,T]}\vert f'_k(x) \vert \right\},
\end{align}
where the last estimate follows from the well-known error bound $((c-a)^2\sup_{x \in [a,c]}\vert g'(x) \vert)/N$ for the absolute difference between an $N$-point, right end-point Riemann sum and the integral of a $C^1$ function $g$ over $[a,c]$. Then, by the continuity of $\Psi_X$, for any $\epsilon > 0$ we may choose an appropriately large $N$ such that
\begin{align*}
\left \vert \frac{1}{N}\sum_{n=0}^{N-1}\psi_X(\langle \mathbf{z}, \mathbf{t}_{n,\,N} \rangle ) - \frac{1}{N}\sum_{n=0}^{N-1}\psi_X(\langle \mathbf{z}, \mathbf{s}_{n,\,N} \rangle )\right\vert &\leq \frac{1}{N}\sum_{n=0}^{N-1}\vert\psi_X(\langle \mathbf{z}, \mathbf{t}_{n,\,N} \rangle ) -  \psi_X(\langle \mathbf{z}, \mathbf{t}_{n,\,N} \rangle )\vert \leq \epsilon.
\end{align*}
This proves \eqref{eq:withb} and therefore also \eqref{eq:bert} for $C^{1}$ functions. \\ \\
To establish the infinite divisibility of $\xi$ we note that \eqref{eq:withb} shows that $\Psi_{\xi^{(b)}} = b\Psi_{\xi^{(1)}} = b\Psi_{\xi}$ and that $e^{-b\Psi_{\xi}}$ is therefore a positive definite function for every $b$ since it is the characteristic function of the random vector $\xi^{(b)}$. Positive definiteness follows from Bochner's Theorem (see for example Theorem 2.13 in \cite{levtype}). Also, we clearly have $\Psi_{\xi}(\mathbf{0}) = 0$ since $\Psi_{X}(0) = 0$. By Theorem 2.15 in \cite{levtype} these two points combined show that $\Psi_{\xi}$ is the characteristic exponent of an ID probability distribution, and hence $\xi$ is an ID random vector.\\ \\
Now one can extend the lemma to $L^{1}$ functions $\{f_k\}_{k = 1}^d$ by exploiting the density of $C^{1}([0,T])$ in $L^{1}([0,T])$. In particular, for each $k$ we can find a sequence of $C^{1}$ functions $\{f_{n,k}\}_{n\geq 1}$ which converges in $L^1$ to $f_k$. Then,
\begin{align*}
\vert u_k(t) - u_{n,k}(t) \vert = \left\vert \int_t^T f_k(t)\d t - \int_t^T f_{n,k}(t)\d t\right\vert \leq \int_0^T\left\vert f_k(t) - f_{n,k}(t)\right\vert\d t
\end{align*}
showing that $u_{n,k} \rightarrow u_k$ uniformly in $t$. This shows that for each $\mathbf{z}$ the functions $\{\Psi_{X}(\langle \mathbf{z},\mathbf{u}_n(\cdot)\rangle)\}_{n\geq 1}$, with $\mathbf{u}_n := (u_{n,1},\cdots,u_{n,d})^{\textnormal{\textbf{T}}}$, are uniformly bounded on $[0,T]$, so that the dominated convergence theorem applies and we have 
\begin{align}\label{eq:rhs}
\lim_{n \rightarrow \infty}\exp\left(-\int_0^T\Psi_X\left(\langle \mathbf{z}, \mathbf{u}_n(t) \rangle\right)\d t\right) = \exp\left(-\int_0^T\Psi_X\left(\langle \mathbf{z}, \mathbf{u}(t) \rangle\right)\d t\right).
\end{align}
On the other hand, $X$ is a.s. bounded on $[0,T]$, so that
\begin{align*}
\lim_{n\rightarrow\infty}\vert \xi_k - \xi_{n,k} \vert = \lim_{n\rightarrow\infty}\left\vert \int_0^{T}X_tf_{k}(t)\d t - \int_0^{T}X_t f_{n,k}(t)\d t \right \vert  \leq\left(\sup_{t\in[0,T]}\vert X_t \vert\right)  \lim_{n\rightarrow\infty}\int_0^{T}\vert f_{k}(t) -   f_{n,k}(t) \vert\d t = 0, 
\end{align*}
a.s.. Therefore $\Xi_n := (\xi_{n,1},\cdots,\xi_{n,d})^{\textnormal{\textbf{T}}}$ converges a.s. and consequently also in distribution to $\xi$. Together with \eqref{eq:rhs}, this implies that for each $\mathbf{z}$
\begin{align}\label{eq:final}
\lim_{n\rightarrow\infty}\e[e^{i\langle \mathbf{z},\Xi_n \rangle}] = \e[e^{i\langle \mathbf{z},\xi \rangle}] = \exp\left(-\int_0^T\Psi_X\left(\langle \mathbf{z}, \mathbf{u}(t) \rangle\right)\d t\right).
\end{align}
Therefore, \eqref{eq:bert} is also proven for functions in $L^1$. Since each $\Xi_n$ has an ID distribution Lemma 3.1.6 in \cite{Messer} guarantees that $\xi$ is also an ID random vector.
\end{proof}
\end{appendices}
\bibliographystyle{plain}
\bibliography{biblio}
\end{document}